\documentclass{article}

\usepackage{amsmath, amssymb, amsthm, hyperref, microtype, paralist, tikz}

\newtheorem{theorem}{Theorem}[section]
\newtheorem{definition}[theorem]{Definition}
\newtheorem{lemma}[theorem]{Lemma}
\newtheorem{proposition}[theorem]{Proposition}
\newtheorem{corollary}[theorem]{Corollary}
\hypersetup{colorlinks=true, citecolor=blue, linkcolor=blue}
\setdefaultenum{\normalfont (1)}{}{}{}
\tikzset{node distance=50pt, auto}

\title{A Prime Decomposition of Probabilistic Automata}

\author{
	Gunnar Carlsson \\
	Department of Mathematics \\
	Stanford University
	\and
	Jun Yu \\
	Institute for Computational \& Mathematical Engineering \\
	Stanford University
}

\date{}

\begin{document}

\maketitle

\tableofcontents

\let\thefootnote\relax\footnote{This paper is based on the second author's doctoral thesis written under the direction of the first author.}

\section{Introduction}

Krohn-Rhodes theorem asserts that every deterministic automaton can be decomposed into cascades of irreducible automata. Algebraically, this implies that a finite semigroup acting on a finite set factors into a finite wreath product of finite simple groups and a semigroup of order $3$ consisting of the identity map and constant maps on a set of order $2$. The semigroups in this factorization are prime under the semidirect product.

In Section \ref{sec:2}, we formulate a definition of probabilistic automata in which a statement analogous to the prime decomposition follows directly from Krohn-Rhodes theorem.

Section \ref{sec:3} deals with Green-Rees theory. We determine Green's relations on the monoid of stochastic matrices in order to characterize the local structure of probabilistic automata.

Krohn-Rhodes theory is introduced in Section \ref{sec:4}. The prime decomposition is presented as a framework to study the global structure of probabilistic automata.

Section \ref{sec:5} discusses Munn-Ponizovski\u{\i} theory. We prove that irreducible representations of a probabilistic automaton are determined by those of finite groups in its holonomy decomposition, which is a variant of the prime decomposition.

\section{Automata and Semigroups} \label{sec:2}

\subsection{Deterministic Automata}

Given a set $X$, $F_X$ denotes the monoid of all maps $X \to X$. If $X$ is of order $n$, we can index $X$ by
	\begin{equation*}
		\mathbf{n} = \{i \mid 0 \leq i < n\}
	\end{equation*}
with a bijection $X \to \mathbf{n}$, and write $F_n \cong F_X$.

\begin{definition}
A deterministic automaton is a triple $(X, \Sigma, \delta)$ consisting of finite sets $X$ and $\Sigma$ along with a map $\delta: X \times \Sigma \to X$. We call $X$ a state set, $\Sigma$ an alphabet, and $\delta$ a transition function.
\end{definition}

Let $\Sigma^*$ be the free monoid on $\Sigma$. We can define a right action of $\Sigma^*$ on $X$ by $xa = \delta(x, a)$, where $x \in X$ and $a \in A$. This action may not be faithful, and hence we consider the canonical homomorphism $\sigma: \Sigma^* \to F_X$. If $\Sigma^+$ is the free semigroup on $A$, then
	\begin{equation*}
		S = \Sigma^+\sigma
	\end{equation*}
acts faithfully on $X$. Since $F_X$ is finite, so is $S$.

\begin{definition}
A transformation semigroup is a pair $(X, S)$ in which a finite semigroup $S$ acts faithfully on $X$ from the right.
\end{definition}

In case $S$ is a monoid such that $1_S = 1_X$, we refer to $(X, S)$ as a \emph{transformation monoid}. If, in addition, $S$ is a group, $(X, S)$ is called a \emph{transformation group}.

If $S$ is not a monoid, we can adjoin an identity element $1$ in a natural way to form a monoid $S^1$. It is understood that $S^1 = S$ when $S$ is a monoid. Similarly, in its absence, adjuction of a zero element $0$ defines a new semigroup $S^0$. We write $\mathbf{FSgp}$ for the category of finite semigroups.

\subsection{Probabilistic Automata}

Let $X$ by a finite set. Then $\mathbb{P}X$ is the set of all probability distributions on $X$. An element $\mu \in \mathbb{P}X$ is written as a formal sum
	\begin{equation*}
		\mu = \sum_{x \in X} \mu(x)x.
	\end{equation*}
We can regard $\mathbb{P}X$ as a subset of the free $\mathbb{R}$-module on $X$, although $\mathbb{P}X$ itself does not have an additive structure.

\begin{definition}
A probabilistic automaton is a quadruple $(X, \Sigma, \delta, \mathbb{P})$ consisting of finite sets $X$ and $\Sigma$ along with a map $\delta: X \times \Sigma \to X$ and its extension $\mathbb{P}\delta: \mathbb{P}X \times \mathbb{P}\Sigma \to \mathbb{P}X$ defined by
	\begin{equation*}
		\mathbb{P}\delta(\pi, \mu) = \sum_{(x, a) \in X \times \Sigma}\pi(x)\mu(a)\delta(x, a)
	\end{equation*}
for $\pi \in \mathbb{P}X$ and $\mu \in \mathbb{P}\Sigma$.
\end{definition}

For a subset $\Omega$ of $\mathbb{P}\Sigma$, the quintuple $(X, \Sigma, \delta, \mathbb{P}, \Omega)$ is an \emph{instance} of $(X, \Sigma, \delta, \mathbb{P})$, in which case $\mathbb{P}\delta$ is restricted to $\mathbb{P}X \times \Omega'$, where $\Omega'$ denotes the closure of the set generated by $\Omega$. When $\Omega$ is finite, $(X, \Sigma, \delta, \mathbb{P}, \Omega)$ resembles the classical definition of a probabilistic automaton \cite{R1}.

Again, set $S = \Sigma^+\sigma$, where $\sigma: \Sigma^* \to F_X$ is the canonical homomorphism. Given $\mu \in \mathbb{P}A$, we abuse notation by writing $\mu$ for its corresponding distribution in $\mathbb{P}S$, so that for any $s \in S$,
	\begin{equation*}
		\mu(s) = \sum_{a\sigma = s} \mu(a).
	\end{equation*}
Then $\mathbb{P}S$ is closed under convolution, which is given by
	\begin{equation*}
		(\mu \ast \nu)(s) = \sum_{s = tu} \mu(t)\nu(u)
	\end{equation*}
for $\mu, \nu \in \mathbb{P}S$, and hence $\mathbb{P}S$ forms a semigroup under convolution. Since $S$ is finite, as a topological semigroup, $\mathbb{P}S$ is compact Hausdorff.

\begin{definition}
A transition semigroup is a triple $(X, S, \mathbb{P})$ in which $S$ is a finite semigroup acting faithfully on a finite set $X$ from the right, inducing a right action of $\mathbb{P}S$ on $\mathbb{P}X$ defined by
		\begin{equation*}
			\pi\mu = \sum_{xs = y} \pi(x)\mu(s)y
		\end{equation*}
for $\pi \in \mathbb{P}X$ and $\mu \in \mathbb{P}S$.
\end{definition}

For $Q \subset \mathbb{P}S$, the quadruple $(X, S, \mathbb{P}, Q)$ is an \emph{instance} of $(X, S, \mathbb{P})$, in which case the action of $\mathbb{P}S$ on $\mathbb{P}X$ is restricted to $Q'$, where $Q'$ denotes the closure of the set generated by $Q$.

It is easy to see that $\pi\mu \in \mathbb{P}X$. Although we require that $S$ acts faithfully on $X$, the same is not true of the action of $\mathbb{P}S$ on $\mathbb{P}X$. We refer to $(X, S, \mathbb{P})$ as a \emph{transition monoid} if $(X, S)$ is a transformation monoid. A \emph{transition group} is defined accordingly.

\section{Local Structure of Probabilistic Automata} \label{sec:3}

\subsection{Green-Rees Theory}

We introduce the work of Green and Rees as presented by Clifford \& Preston \cite{CP} and Rhodes \& Steinberg \cite{RS}.

A subset $I \neq \emptyset$ of a semigroup $S$ is a \emph{left ideal} if $SI \subset I$. A \emph{right ideal} is defined dually. We say $I$ is an \emph{ideal} if it is both a left and right ideal. Moreover, $S$ is \emph{left simple}, \emph{right simple}, or \emph{simple} if it does not contain a proper left ideal, right ideal, or ideal. For any $s \in S$, we refer to $L(s) = S^1s$, $R(s) = sS^1$, and $J(s) = S^1sS^1$, respectively, as the \emph{principal left ideal}, \emph{principal right ideal}, and \emph{principal ideal} generated by $s$.

\begin{definition}
Let $S$ be a semigroup. Then the quasiorders on $S$ given by
\begin{compactenum}
	\item $s \leq_\mathfrak{l} t$ if and only if $L(s) \subset L(t)$,
	\item $s \leq_\mathfrak{r} t$ if and only if $R(s) \subset R(t)$,
	\item $s \leq_\mathfrak{j} t$ if and only if $J(s) \subset J(t)$,
	\item $s \leq_\mathfrak{h} t$ if and only if $s \leq_\mathfrak{l} t$ and $s \leq_\mathfrak{r} t$
\end{compactenum}
induce equivalence relations $\sim_\mathfrak{l}$, $\sim_\mathfrak{r}$, $\sim_\mathfrak{h}$, and $\sim_\mathfrak{j}$, respectively, on $S$. Furthermore, the relation
	\begin{equation*}
		\mathfrak{d} = \mathfrak{l} \circ \mathfrak{r} = \mathfrak{r} \circ \mathfrak{l}
	\end{equation*}
in $S \times S$ defines an equivalence relation $\sim_\mathfrak{d}$ on $S$. These five equivalence relations on $S$ are known as Green's relations.
\end{definition}

Green's relations coincide in a commutative semigroup, while each relation is trivial for a group. In $S \times S$,
	\begin{equation*}
		\mathfrak{h} = \mathfrak{l} \cap \mathfrak{r} \subset \mathfrak{l} \cup \mathfrak{r} \subset \mathfrak{d} \subset \mathfrak{j}.
	\end{equation*}
Moreover, $\sim_\mathfrak{l}$ is a right congruence and $\sim_\mathfrak{r}$ is a left congruence. We write the $\mathfrak{l}$-class of $s \in S$ as
	\begin{equation*}
		L_s = \{t \in S \mid s \sim_\mathfrak{l} t \},
	\end{equation*}
and define $R_s$, $J_s$, $H_s$, and $D_s$ analogously.

\begin{proposition} \label{prop:Identities}
If $e$ is an idempotent in a semigroup $S$, then
	\begin{inparaenum}
		\item $Se \cap J_e = L_e$,
		\item $eS \cap J_e = R_e$, and
		\item $eSe \cap J_e = H_e$.
	\end{inparaenum}
\end{proposition}

For any $u \in S$, the \emph{left translation} by $u$ is the map $\lambda_u: S \to S$ defined by $s\lambda_u = us$. Its dual, denoted $\rho_u$, is the \emph{right translation} by $u$. Green \cite{G1} used translations to construct bijections $L_s \to L_t$ and $R_s \to R_t$ when $s \sim_\mathfrak{d} t$.

\begin{lemma}[Green] \label{lem:Green} Suppose $s, t \in S$, where $S$ is a semigroup.
	\begin{compactenum}
		\item If $us = t$ and $vt = s$ for $u, v \in S^1$, so that $s \sim_\mathfrak{l} t$, then the maps $\lambda_u|_{R_s}$ and $\lambda_v|_{R_t}$ are inverses of one another.
		\item If $su = t$ and $tv = s$ for $u, v \in S^1$, so that $s \sim_\mathfrak{r} t$, then the maps $\rho_u|_{L_s}$ and $\rho_v|_{L_t}$ are inverses of one another.
	\end{compactenum}
\end{lemma}

Koch \& Wallace \cite{KW} formulated a sufficient condition for $\mathfrak{d}$- and $\mathfrak{j}$-relations to agree with one another. A semigroup $S$ is said to be \emph{stable} if
\begin{compactenum}
	\item $s \sim_\mathfrak{l} ts$ if and only if $s \sim_\mathfrak{j} ts$,
	\item $s \sim_\mathfrak{r} st$ if and only if $s \sim_\mathfrak{j} st$
\end{compactenum}
for any $s, t \in S$. This ensures that $D_s = J_s$ for every $s \in S$. In particular, finite semigroups, commutative semigroups, and compact semigroups are stable. For stable semigroups, Lemma \ref{lem:Green} implies that $\mathfrak{l}$-classes contained in the same $\mathfrak{j}$-class have identical cardinality. The same is true of $\mathfrak{r}$- and $\mathfrak{h}$-classes.

We say $s \in S$ is \emph{regular}, in the sense of von Neumann, if there exists $t \in S$ such that $sts = s$. If, in addition, $tst = t$, $t$ is an \emph{inverse} of $s$. A regular element always has an inverse, and so $s$ is regular if and only if $s$ has an inverse. We call $S$ a \emph{regular semigroup} if each of its elements are regular. If every element has a unique inverse, then $S$ is an \emph{inverse semigroup}.

\begin{definition}
Given sets $\Lambda$ and $\Gamma$, a $\Lambda \times \Gamma$ Rees matrix over a group $G$ is a map $(u_{\lambda\rho}): \Lambda \times \Gamma \to G$. A Rees semigroup of matrix type is a set
	\begin{equation*}
		\mathfrak{M}(G, \Gamma, \Lambda, (u_{\lambda\rho})) = \{(\rho, g, \lambda) \mid g \in G, \rho \in \Gamma, \lambda \in \Lambda\}
	\end{equation*}
endowed with a product defined by the rule
	\begin{equation*}
		(\rho, g, \lambda)(\gamma, h, \alpha) = (\rho, gu_{\lambda\gamma}h, \alpha).
	\end{equation*}
We call $G$ the structure group of $\mathfrak{M}(G, \Gamma, \Lambda, (u_{\lambda\rho}))$.
\end{definition}

It is easy to see that $\mathfrak{M}(G, \Gamma, \Lambda, (u_{\lambda\rho}))$ is indeed a semigroup. By convention, we write
	\begin{equation*}
		\mathfrak{M}^0(G, \Gamma, \Lambda, (u_{\lambda\rho})) = \mathfrak{M}(G^0, \Gamma, \Lambda, (u_{\lambda\rho})).
	\end{equation*}
Moreover, $(u_{\lambda\rho})$ is called \emph{regular} if every row and column has a nonzero entry, which is the same as saying $\mathfrak{M}^0(G, \Gamma, \Lambda, (u_{\lambda\rho}))$ is regular as a semigroup.

Suppose $0 \in S$ and $S^2 \neq 0$. Then $S$ said to be \emph{$0$-simple} if it does not contain a nonzero proper ideal. It is easy to see that if $0 \notin S$, then $S$ is simple if and only if $S^0$ is $0$-simple. Under the stability assumption, Rees \cite{R2} classified $0$-simple semigroups in terms of Rees matrices.

\begin{theorem}[Rees] \label{thm:Rees}
A stable semigroup $S$ is $0$-simple if and only if
	\begin{equation*}
		S \cong \mathfrak{M}^0(G, \Gamma, \Lambda, (u_{\lambda\rho}))
	\end{equation*}
such that $G$ is a group and $(u_{\lambda\rho})$ is regular.
\end{theorem}

Assume $S$ is stable. If $s \in S$ is regular, then every element of $J_s$ is regular. Moreover, there exists an idempotent $e \in J_s$ such that $H_e$ is a maximal subgroup of $S$ with $e$ as identity, and $H_e \cong H_f$ for any idempotent $f \in J_s$.

For every $s \in S$, set $I(s) = J(s) - J_s$. Then $I(s)$ is an ideal of $J(s)$ unless it is empty. The \emph{principal factor} of $S$ at $s$ is the semigroup
	\begin{equation*}
		J_s^0 = \begin{cases}
			J(s)/I(s) & \text{if } J_s \text{ is not the minimal ideal}, \\
			J_s \cup 0 & \text{otherwise}.
		\end{cases}
	\end{equation*}
Alternatively, we can think of $J_s^0$ as the set $J_s \cup 0$ endowed with a product given by the rule
	\begin{equation*}
		tu = \begin{cases}
			tu & \text{if } tu \in J_s, \\
			0 & \text{otherwise}.
		\end{cases}
	\end{equation*}
If $S$ is stable, $J_s$ is regular if and only if $J_s^0$ is $0$-simple, in which case, by Theorem \ref{thm:Rees}, there is an isomorphism $J_s^0 \to \mathfrak{M}^0(G, \Gamma, \Lambda, (u_{\lambda\rho}))$. If $J_s$ is nonregular, then $J_s^0$ is a \emph{null semigroup} in which $tu = 0$ for all $t, u \in J_s$.

\subsection{Local Structure of Transition Semigroups}

Any matrix over $\mathbb{R}$ is said to be \textit{stochastic} if all entries are nonnegative and each row sums to unity. We write $\mathrm{S}(n, \mathbb{R})$ for the monoid of $n \times n$ stochastic matrices over $\mathbb{R}$. A stochastic matrix is \textit{bistochastic} if each column sums to unity. The submonoid of bistochastic matrices in $\mathrm{S}(n, \mathbb{R})$ is denoted $\mathrm{B}(n, \mathbb{R})$. We can also define a stochastic matrix over any proper unitary subring of $\mathbb{R}$. In particular, $\mathrm{S}(n, \mathbb{Z})$ is the monoid of maps $\mathbf{n} \to \mathbf{n}$ and $\mathrm{B}(n, \mathbb{Z})$ is the group of permutations on $\mathbf{n}$.

We associate with each $s \in S$ a matrix $(s_{xy}): X \times X \to [0, 1]$ with $(x, y) \mapsto \delta_{xs}^{y}$, where $\delta_x^y$ is the Kronecker delta on $X \times X$. Clearly, $(s_{xy})$ is row monomial, and hence
	\begin{equation*}
		(\mu_{xy}) = \sum_{s \in S} \mu(s) \cdot (s_{xy})
	\end{equation*}
is stochastic for any $\mu \in \mathbb{P}S$. It is readily verified that
	\begin{equation*}
		((\mu \ast \nu)_{xy}) = (\mu_{xy})(\nu_{xy}).
	\end{equation*}
For any finite semigroup $S$, $\mathbb{P}S$ is isomorphic to a subsemigroup of $\mathbb{P}F_n \cong \mathrm{S}(n, \mathbb{R})$, and so we first study Green's relations on $\mathrm{S}(n, \mathbb{R})$. Schwarz \cite{S3} showed that every maximal subgroup is isomorphic to a symmtric group $S_k$ for some $1 \leq k \leq n$. Wall \cite{W} characterized $\mathfrak{l}$- and $\mathfrak{r}$-relations for regular elements of $\mathrm{S}(n, \mathbb{R})$. Green's relations on $\mathrm{B}(n, \mathbb{R})$ were resolved by Montague \& Plemmons \cite{MP}.

Let $(s_{ij}) \in \mathrm{S}(n, \mathbb{R})$. In block matrix form, $0$ and $1$, respectively, stand for the zero and identity matrices of suitable size. There exists $(p_{ij}) \in \mathrm{B}(n, \mathbb{Z})$ such that
	\begin{equation*}
		(p_{ij})(s_{ij}) = \begin{pmatrix} s_0^t \\ s_1^t \end{pmatrix},
	\end{equation*}
where rows of $s_0^t$ are linearly independent vectors that generate the same convex cone as rows of $(s_{ij})$. A \emph{row echelon form} of $(s_{ij})$ is any matrix of the form
	\begin{equation*}
		\begin{pmatrix} 1 & 0 \\ u & 0 \end{pmatrix} (p_{ij})(s_{ij}),
	\end{equation*}
where $u$ is stochastic. We call $s_0^t$ a \emph{reduced row echelon form} of $(s_{ij})$, which is unique up to row permutation. A pair of elements of $\mathrm{S}(n, \mathbb{R})$ is \emph{row equivalent} if they have identical reduced row echelon form up to row permuation.

If $(s_{ij})$ has a pair of nonzero columns in the same direction, then they appear as the first two columns of $(s_{ij})(p_{ij})$ for some $(p_{ij}) \in \mathrm{B}(n, \mathbb{Z})$. Their sum, whose direction remains unchanged, is the first column of
	\begin{equation*}
		(s_{ij})(p_{ij}) \begin{pmatrix} e & 0 \\ 0 & 1 \end{pmatrix},
	\end{equation*}
where the leftmost entries of $e \in \mathrm{B}(2, \mathbb{Z})$ are unity. We can repeat this process of adding up columns in the same direction until the matrix is in \emph{column echelon form}
	\begin{equation*}
		\begin{pmatrix} s_0 & s_1 \end{pmatrix},
	\end{equation*}
where nonzero columns are pairwise in different directions and columns of $s_0$, which are linearly independent, generate the same convex cone as columns of $(s_{ij})$. The \emph{reduced column echelon form} of $(s_{ij})$, which is unique up to column permutation, is obtained by removing any zero columns from $a_1$. When a pair of elements of $\mathrm{S}(n, \mathbb{R})$ have identical reduced column echelon form up to column permutation, we say that they are \emph{column equivalent}.

The \emph{echelon form} of $(s_{ij})$ is the row echelon form of the column echelon form of $(s_{ij})$. This is the same as the column echelon form of the row echelon form of $(s_{ij})$ as matrix multiplication is associative. If the \emph{reduced echelon form} is defined accordingly, then it is unique up to row and column permutations. A pair of elements of $\mathrm{S}(n, \mathbb{R})$ is called \emph{equivalent} if they have identical reduced echelon form up to row and column permutations.

\begin{proposition} \label{prop:Stochasticity}
If $(s_{ij}), (t_{ij}) \in \mathrm{S}(n, \mathbb{R})$, then
	\begin{compactenum}
		\item $(s_{ij}) \sim_\mathfrak{l} (t_{ij})$ if and only if $(s_{ij})$ and $(t_{ij})$ are row equivalent,
		\item $(s_{ij}) \sim_\mathfrak{r} (t_{ij})$ if and only if $(s_{ij})$ and $(t_{ij})$ are column equivalent,
		\item $(s_{ij}) \sim_\mathfrak{j} (t_{ij})$ if and only if $(s_{ij})$ and $(t_{ij})$ are equivalent,
		\item $(s_{ij}) \sim_\mathfrak{h} (t_{ij})$ if and only if $(s_{ij})$ and $(t_{ij})$ are row and column equivalent.
	\end{compactenum}
\end{proposition}

\begin{proof}
(1) Suppose $(s_{ij}) \sim_\mathfrak{l} (t_{ij})$. Then the rows of $(s_{ij})$ and $(t_{ij})$ generate the same convex cone, and so they must be row equivalent.

Conversely, if $(s_{ij})$ and $(t_{ij})$ are row equivalent, then there exists $(p_{ij}), (q_{ij}) \in \mathrm{B}(n, \mathbb{Z})$ such that
	\begin{equation*}
		(p_{ij})(s_{ij}) = \begin{pmatrix} s_0^t \\ s_1^t \end{pmatrix}
		\text{ and }
		(q_{ij})(t_{ij}) = \begin{pmatrix} t_0^t \\ t_1^t \end{pmatrix}
	\end{equation*}
are in row echelon form with $rs_0^t = t_0^t$ for some permutation $r$. Moreover, every row of $t_1^t$ is contained in the convex hull generated by the rows of $s_0^t$, so that we can find $u$ that is stochastic and satisfies $us_0^t = t_1^t$. Similarly, $vs_0^t = s_1^t$, where $v$ is stochastic. Therefore
	\begin{equation*}
		(q_{ij})^t \begin{pmatrix} r & 0 \\ u & 0 \end{pmatrix} (p_{ij})(s_{ij}) = (t_{ij})
		\text{ and }
		(p_{ij})^t \begin{pmatrix} r^t & 0 \\ v & 0 \end{pmatrix} (q_{ij})(t_{ij}) = (s_{ij}),
	\end{equation*}
and so we are done.

(2) If the first two columns of $(s_{ij})(p_{ij})$ are in the same direction, then for any $u \in \mathrm{S}(2, \mathbb{R})$ of rank one, we can always find $v \in \mathrm{S}(2, \mathbb{R})$ of rank one such that
	\begin{equation*}
		(s_{ij})(p_{ij}) \begin{pmatrix} u & 0 \\ 0 & 1 \end{pmatrix} \begin{pmatrix} v & 0 \\ 0 & 1 \end{pmatrix} = (s_{ij})(p_{ij}).
	\end{equation*}
This shows that $(s_{ij})$ and its column echelon form are $\mathfrak{r}$-related.

Let $(s_{ij}) \sim_\mathfrak{r} (t_{ij})$. We can assume $(s_{ij})$ and $(t_{ij})$ are in column echelon form. Then there exist $(u_{ij}), (v_{ij}) \in \mathrm{S}(n, \mathbb{R})$ such that
	\begin{equation*}
		\begin{pmatrix} s_0 & s_1 \end{pmatrix} = \begin{pmatrix} t_0 & t_1 \end{pmatrix} \begin{pmatrix} u_{00} & u_{01} \\ u_{10} & u_{11} \end{pmatrix}
		\text{ and }
		\begin{pmatrix} t_0 & t_1 \end{pmatrix} = \begin{pmatrix} s_0 & s_1 \end{pmatrix} \begin{pmatrix} v_{00} & v_{01} \\ v_{10} & v_{11} \end{pmatrix}.
	\end{equation*}
We can now write
	\begin{equation*}
		s_0 = t_0u_{00} + t_1u_{10}.
	\end{equation*}
Columns of $s_0$ generate the same convex cone as those of $t_0$, and hence $s_0 = t_0dp$, where $d$ is diagonal and $p$ a permutation. Furthermore, columns of $t_1$ are properly contained in the convex cone generated by those of $t_0$, so that $t_1 = t_0w$ for some $w$ that has at least two positive entries in every column. This implies that $u_{10} = 0$, whence $t_0(dp - u_{00}) = 0$. As columns of $t_0$ are linearly independent, it follows that $u_{00} = dp$. By a similar reasoning for
	\begin{equation*}
		t_0 = s_0v_{00} + s_1v_{10},
	\end{equation*}
we can deduce that $v_{00} = p^td^{-1}$ and $v_{10} = 0$. This shows $d = 1$, or else $(u_{ij})$ or $(v_{ij})$ fails to be stochastic. It is immediate that $u_{01} = v_{01} = 0$, and so $s_1 = t_1u_{11}$ and $t_1 = s_1v_{11}$. If nonzero columns of $s_1$ and $t_1$ are linearly independent, we are done. Otherwise, we can repeat this argument for $s_1$ and $t_1$. This process ends in finite steps, and thus the result follows.

(3) By stability, $(s_{ij}) \sim_\mathfrak{j} (t_{ij})$ if and only if there exists $(u_{ij}) \in \mathrm{S}(n, \mathbb{R})$ such that $(s_{ij}) \sim_\mathfrak{l} (u_{ij})$ and $(u_{ij}) \sim_\frak{r} (t_{ij})$, which is the same as saying the reduced column echelon form of the reduced row echelon form of $(s_{ij})$ is identical to the reduced column echelon form of the reduced row echelon form of $(t_{ij})$ up to row and column permutations.

(4) This is a direct consequence of (1) and (2).
\end{proof}

Every compact semigroup contains an idempotent, so that $J_\mu$ is regular for some $\mu \in \mathbb{P}S$. Doob \cite{D} identified all idempotent elements in $\mathrm{S}(n, \mathbb{R})$.

\begin{theorem}[Doob] \label{thm:Doob}
If $(e_{ij}) \in \mathrm{S}(n, \mathbb{R})$ is of rank $k$ with $1 \leq k \leq n$, then $(e_{ij})$ is idempotent if and only if there exists $(p_{ij}) \in \mathrm{B}(n, \mathbb{Z})$ such that
	\begin{equation*}
		(p_{ij})(e_{ij})(p_{ij})^t = \begin{pmatrix} e & 0 \\ se & 0 \end{pmatrix},
	\end{equation*}
where $s$ is stochastic and $e$ is of the form
	\begin{equation*}
		e = \begin{pmatrix} e_1 & & \\ & \ddots & \\ & & e_k \end{pmatrix}
	\end{equation*}
such that $e_i$ is rank one and stochastic for $1 \leq i \leq k$.
\end{theorem}

We can count the number of distinct regular $\mathfrak{j}$-classes in $\mathrm{S}(n, \mathbb{R})$ once it is known which idempotent elements belong to the same $\mathfrak{j}$-class.

\begin{corollary} \label{cor:Idempotency}
If $(e_{ij})$ and $(f_{ij})$ are idempotent in $\mathrm{S}(n, \mathbb{R})$, then $(e_{ij}) \sim_\mathfrak{j} (f_{ij})$ if and only if $\mathrm{rank}(e_{ij}) = \mathrm{rank}(f_{ij})$.
\end{corollary}

\begin{proof}
Suppose $(e_{ij})$ is of rank $k$. It follows from Theorem \ref{thm:Doob} that there exists $(p_{ij}) \in \mathrm{B}(n, \mathbb{Z})$ such that the reduced echelon form of $(p_{ij})(e_{ij})(p_{ij})^t$ is an identity in $\mathrm{S}(k, \mathbb{Z})$. This completes the proof.
\end{proof}

It is immediate from Corollary \ref{cor:Idempotency} that there are $n$ regular $\mathfrak{j}$-classes in $\mathrm{S}(n, \mathbb{R})$. In general, we cannot say that if $(e_{ij}) \sim_\mathfrak{j} (f_{ij})$ in $\mathrm{S}(n, \mathbb{R})$, then $(e_{ij}) \sim_\mathfrak{j} (f_{ij})$ in a proper subsemigroup of $\mathrm{S}(n, \mathbb{R})$. Consider, for example, the subsemigroup
	\begin{equation*}
		\left\{ \begin{pmatrix} 1 & 0 & 0 \\ 0 & 0 & 1 \\ 0 & 0 & 1 \end{pmatrix}, \begin{pmatrix} 0 & 0 & 1 \\ 0 & 1 & 0 \\ 0 & 0 & 1 \end{pmatrix}, \begin{pmatrix} 0 & 0 & 1 \\ 0 & 0 & 1 \\ 0 & 0 & 1 \end{pmatrix} \right\}
	\end{equation*}
of $\mathrm{S}(3, \mathbb{R})$. It is true, however, that if $t$ and $u$ are regular in a subsemigroup $T$ of $S$, then $t \sim_\mathfrak{l} u$ in $T$ if and only if $t \sim_\mathfrak{l} u$ in $S$. Analogous statements hold for $\mathfrak{r}$- and $\mathfrak{h}$-relations.

\begin{theorem} \label{thm:Coordinate}
Suppose $(X, S, \mathbb{P})$ is a transition semigroup such that $\varphi: \mathbb{P}S \to T$ is an isomorphism, where $n = |X|$ and $T$ is a subsemigroup of $\mathrm{S}(n, \mathbb{R})$. For any idempotent $e \in \mathbb{P}S$, define $\Lambda = \{\lambda \in T \mid \lambda \sim_\mathfrak{r} e\varphi\}$ and $\Gamma = \{\rho \in T \mid \rho \sim_\mathfrak{l} e\varphi\}$. If $G = H_{e\varphi}$, then
	\begin{equation*}
		J_e^0 \cong \mathfrak{M}^0(G, \Gamma, \Lambda, (u_{\lambda\rho})),
	\end{equation*}
where $(u_{\lambda\rho}): \Lambda \times \Gamma \to G^0$ is given by
	\begin{equation*}
		u_{\lambda\rho} = \begin{cases}
			\lambda\rho & \text{if } \lambda\rho \in G, \\
			0 & \text{otherwise}.
		\end{cases}
	\end{equation*}
Here, $(\rho,g,\lambda) = 0$ in $\mathfrak{M}^0(G, \Gamma, \Lambda, (u_{\lambda\rho}))$ whenever $g = 0$.
\end{theorem}

\begin{proof}
This follows directly from Theorem \ref{thm:Rees} and Proposition \ref{prop:Stochasticity}.
\end{proof}

Theorem \ref{thm:Coordinate} carries over to an instance $(X, S, \mathbb{P}, Q)$ of $(X, S, \mathbb{P})$ since $Q'$ is compact, and hence stable.

\section{Global Structure of Probabilistic Automata} \label{sec:4}

\subsection{Krohn-Rhodes Theory}

A pair of transformation semigroups $(X, S)$ and $(Y, T)$ are said to be \emph{isomorphic}, written $(X, S) \cong (Y, T)$, if there exists a bijective map $\varphi: Y \to X$ such that
	\begin{compactenum}
		\item $\varphi{s}\varphi^{-1} \in T$ for all $s \in S$,
		\item $\varphi^{-1}t\varphi \in S$ for all $t \in T$.
	\end{compactenum}
It is easy to see that this implies $S$ is isomorphic to $T$.

\begin{definition}
Let $(X, S)$ and $(Y, T)$ be transformation semigroups. If there exists a surjective partial map $\varphi: Y \to X$ such that for every $s \in S$, $\varphi{s} = t\varphi$ for some $t \in T$, so that the diagram
	\begin{center}
		\begin{tikzpicture}
			\node (0) {$Y$};
			\node (1) [below of=0] {$X$};
			\node (2) [right of=1] {$X$};
			\node (3) [above of=2] {$Y$};
			\draw[->] (0) to node [swap] {$\varphi$} (1);
			\draw[->] (1) to node [swap] {$s$} (2);
			\draw[->] (3) to node {$\varphi$} (2);
			\draw[->] (0) to node {$t$} (3);
		\end{tikzpicture}
	\end{center}
commutes, then $(X, S)$ is said to divide $(Y, T)$ by $\varphi$. We write
	\begin{equation*}
		(X, S) \prec (Y, T)
	\end{equation*}
to mean $(X, S)$ is a divisor of $(Y, T)$, and refer to $\varphi$
as a covering.
\end{definition}

If $T$ is not a monoid, a homomorphism $\varphi: T \to S$ has a natural extension $\varphi^1: T^1 \to S^1$ given by
	\begin{equation*}
		t\varphi^1 =
		\begin{cases}
			1 & \text{if } t = 1, \\
			t\varphi & \text{otherwise}.
		\end{cases}
	\end{equation*}
In case $T$ is a monoid, set $\varphi^1 = \varphi$. We often identify $S$ with the transformation semigroup $(S^1, S)$, and say that $T$ covers $S$ when there is a covering $\varphi^1$, so that $T$ covers $S$ as transformation semigroups.

If $x \in X$, $\bar{x}$ stands for the constant map $X \to X$ onto $x$. The semigroup of all such maps is denoted $\bar{X}$. The \emph{closure} of $(X, S)$ is the transformation semigroup
	\begin{equation*}
		\overline{(X, S)} = (X, S \cup \bar{X}).
	\end{equation*}
As the empty set is vacuously a semigroup, $X$ can be identified with the transformation semigroup $(X, \emptyset)$, in which case $\bar{X} = (X, \bar{X})$. In addition, we associate to $(X, S)$ the transformation monoid
	\begin{equation*}
		(X, S)^1 = (X, S \cup 1_X),
	\end{equation*}
which means $S^1 = (S^1, S^1)$.

\begin{definition} \label{def:WreathProduct}
Let $(X, S)$ and $(Y, T)$ be transformation semigroups. Suppose that the action of $t \in T$ on $f \in S^Y$ is given by $y{\,^t\!f} = ytf$ for any $y \in Y$. Then the wreath product of $(X, S)$ by $(Y, T)$ is the transformation semigroup
	\begin{equation*}
		(X, S) \wr (Y, T) = (X \times Y, S^Y \rtimes T),
	\end{equation*}
where $(x, y)(f, t) = (x(yf), yt)$ for any $(x, y) \in X \times Y$ and $(f, t) \in S^Y \rtimes T$.
\end{definition}

Let $\mathbf{TSgp}$ denote the category in which objects are transformation semigroups and morphisms are coverings of objects. Evidently, $(X, S) \cong (Y, T)$ if and only if $(X, S) \prec (Y, T)$ and $(Y, T) \prec (X, S)$, whence $\prec$ is a partial order on $\mathbf{TSgp}$. In Definition \ref{def:WreathProduct}, it is routine to check that $S^Y \rtimes T$ is a semigroup acting faithfully on $X \times Y$. It follows that isomorphism classes of $\mathbf{TSgp}$ form a monoid under the binary operation $\wr$ with unity $\mathbf{1}^1$. A \emph{decomposition} of $(X, S)$ is an inequality in $\mathbf{TSgp}$ of the form
	\begin{equation*}
		(X, S) \prec (X_1, S_1) \wr \cdots \wr (X_n, S_n)
	\end{equation*}
such that either $X_i$ is strictly smaller than $X$ or $S_i$ is strictly smaller than $S$ for all $1 \leq i \leq n$.

\begin{proposition} \label{prop:Decompositions}
Let $(X, S)$ be a transformation semigroup.
	\begin{compactenum}
		\item If $G$ is a maximal subgroup of $S$, then
			\begin{equation*}
				(X, S) \prec (X, S \backslash G)^1 \wr G.
			\end{equation*}
		\item If $S = I \cup T$, where $I$ is a left ideal in $S$ and $T$ a subsemigroup of $S$, then
			\begin{equation*}
				(X, S) \prec (X, I)^1 \wr \overline{(T \cup 1_X, T)}.
			\end{equation*}
	\end{compactenum}
\end{proposition}

Every finite group admits a composition series, which determines a unique collection of simple group divisors. Jordan-H\"older decomposition accounts for all simple group divisors.

\begin{theorem}[Jordan-H\"older] \label{thm:JordanHolder}
If $G$ is a finite group, then
	\begin{equation*}
		G \prec G_1 \wr \cdots \wr G_n,
	\end{equation*}
where $G_i$ is a simple group divisor of $G$ for $1 \leq i \leq n$.
\end{theorem}

By Proposition \ref{prop:Decompositions}, we can view Theorem \ref{thm:JordanHolder} as a decomposition for transformation groups. Krohn-Rhodes decomposition generalizes Jordan-H\"older decomposition to transformation semigroups. Krohn and Rhodes \cite{KR} first showed that a finite semigroup is either cyclic, left simple, or the union of a proper left ideal and a proper subsemigroup, and then argued inductively by showing that any transformation semigroup admits a decomposition in $\mathbf{TSgp}$.

\begin{theorem}[Krohn-Rhodes] \label{thm:KrohnRhodes}
If $(X, S)$ is a transformation semigroup, then
	\begin{equation*}
		(X, S) \prec (X_1, S_1) \wr \cdots \wr (X_n, S_n),
	\end{equation*}
where either $(X_i, S_i) = \overline{\mathbf{2}}{^1}$ or $(X_i, S_i)$ is a simple group divisor of $S$ for $1 \leq i \leq n$.
\end{theorem}

In $\mathbf{FSgp}$, we say $S$ is \emph{prime} if $S \prec T \rtimes U$ implies that either $S \prec T$ or $S \prec U$. The prime semigroups are precisely the divisors of $\overline{\mathbf{2}}{^1}$ and the finite simple groups. The decomposition of Theorem \ref{thm:KrohnRhodes} is called the \emph{prime decomposition}.

\subsection{Global Structure of Transition Semigroups}

Let $X$ and $Y$ be finite sets. If $\varphi: Y \to X$ is a partial map, we define its extension to be a partial map $\mathbb{P}\varphi: \mathbb{P}Y \to \mathbb{P}X$ given by
	\begin{equation*}
		\pi(\mathbb{P}\varphi) =
		\begin{cases}
			\displaystyle \sum_{x \in X}\sum_{y\varphi = x} \pi(y)x & \text{if } y\varphi \neq \emptyset \text{ whenever } \pi(y) > 0, \\
			\emptyset & \text{otherwise}
		\end{cases}
	\end{equation*}
for any $\pi \in \mathbb{P}Y$.

\begin{definition}
Let $(X, S, \mathbb{P})$ and $(Y, T, \mathbb{P})$ be transition semigroups. If there exists a surjective partial map $\varphi: Y \to X$ with extension $\mathbb{P}\varphi: \mathbb{P}Y \to \mathbb{P}X$ such that for every $\mu \in \mathbb{P}S$, $(\mathbb{P}\varphi)\mu = \nu(\mathbb{P}\varphi)$ for some $\nu \in \mathbb{P}T$, so that the diagram
	\begin{center}
		\begin{tikzpicture}
			\node (0) {$\mathbb{P}Y$};
			\node (1) [below of=0] {$\mathbb{P}X$};
			\node (2) [right of=1] {$\mathbb{P}X$};
			\node (3) [above of=2] {$\mathbb{P}Y$};
			\draw[->] (0) to node [swap] {$\mathbb{P}\varphi$} (1);
			\draw[->] (1) to node [swap] {$\mu$} (2);
			\draw[->] (3) to node {$\mathbb{P}\varphi$} (2);
			\draw[->] (0) to node {$\nu$} (3);
		\end{tikzpicture}
	\end{center}
commutes, then $(X, S, \mathbb{P})$ is said to divide $(Y, T, \mathbb{P})$ by $\mathbb{P}\varphi$. We write
	\begin{equation*}
		(X, S, \mathbb{P}) \prec (Y, T, \mathbb{P})
	\end{equation*}
to mean $(X, S, \mathbb{P})$ is a divisor of $(Y, T, \mathbb{P})$, and refer to $\varphi$ as a covering.
\end{definition}

Notation for transformation semigroups naturally carry over to transition semigroups. Therefore
	\begin{equation*}
		\overline{(X, S, \mathbb{P})} = (X, S \cup \bar{X}, \mathbb{P}) \text{ and } (X, S, \mathbb{P})^1 = (X, S \cup 1_X, \mathbb{P}).
	\end{equation*}
We also identify $(X, \mathbb{P})$ with $(X, \emptyset, \mathbb{P})$ and $(S, \mathbb{P})$ with $(S^1, S, \mathbb{P})$.

\begin{lemma} \label{lem:Duality}
If $(X, S, \mathbb{P})$ and $(Y, T, \mathbb{P})$ are transition semigroups, then $(X, S, \mathbb{P})$ divides $(Y, T, \mathbb{P})$ if and only if $(X, S)$ divides $(Y, T)$.
\end{lemma}

\begin{proof}
Suppose $(X, S, \mathbb{P})$ divides $(Y, T, \mathbb{P})$ by $\mathbb{P}\varphi$. Fix $s \in S$. Then $(\mathbb{P}\varphi)s = \nu(\mathbb{P}\varphi)$ for some $\nu \in \mathbb{P}Y$. This means
	\begin{equation*}
		y\varphi{s} = \sum_{t \in T} \nu(t)yt\varphi
	\end{equation*}
for any $y \in Y$ such that $y\varphi \neq \emptyset$. We conclude $\varphi{s} = t\varphi$ for some $t \in T$ with $\nu(t) > 0$.

Conversely, assume $(X, S)$ divides $(Y, T)$ by $\varphi$. Given $\mu \in \mathbb{P}S$, choose $t \in T$ such that $\varphi{s} = t\varphi$ for every $s \in S$ with $\mu(s) > 0$. Let $U \subset T$ be the collection of all such selections. Define $\nu \in \mathbb{P}T$ by
	\begin{equation*}
		\nu(t) =
		\begin{cases}
			\displaystyle \sum_{\varphi{s} = t\varphi} \mu(s) & \text{if } t \in U, \\
			0 & \text{otherwise}.
		\end{cases}
	\end{equation*}
Then we can write
	\begin{equation*}
		\pi(\mathbb{P}\varphi)\mu = \sum_{x \in X}\sum_{y\varphi{s} = x} \pi(y)\mu(s)x = \sum_{x \in X}\sum_{yt\varphi = x} \pi(y)\nu(t)x = \pi\nu(\mathbb{P}\varphi),
	\end{equation*}
where $\pi \in \mathbb{P}Y$.
\end{proof}

To extend Definition \ref{def:WreathProduct} to transition semigroups, we take the wreath product of $(X, S)$ by $(Y, T)$, and consider the right action of $\mathbb{P}(S^Y \rtimes T)$ on $\mathbb{P}(X \times Y)$.

\begin{definition}
Let $(X, S, \mathbb{P})$ and $(Y, T, \mathbb{P})$ be transition semigroups. The wreath product of $(X, S, \mathbb{P})$ by $(Y, T, \mathbb{P})$ is the transition semigroup
	\begin{equation*}
		(X, S, \mathbb{P}) \wr (Y, T, \mathbb{P}) = (Z, U, \mathbb{P}),
	\end{equation*}
where $(Z, U) = (X, S) \wr (Y, T)$.
\end{definition}

It is clear that $(Z, U, \mathbb{P})$ is well-defined since $(X, S) \wr (Y, T)$ is a transformation semigroup in its own right.

\begin{theorem} \label{thm:Prime}
If $(X, S, \mathbb{P})$ is a transition semigroup, then
	\begin{equation*}
		(X, S, \mathbb{P}) \prec (X_1, S_1, \mathbb{P}) \wr \cdots \wr (X_n, S_n, \mathbb{P}),
	\end{equation*}
where either $(X_i, S_i) = \overline{\mathbf{2}}{^1}$ or $(X_i, S_i)$ is a simple group divisor of $S$ for $1 \leq i \leq n$.
\end{theorem}

\begin{proof}
This is an immediate consequence of Theorem \ref{thm:KrohnRhodes} and Lemma \ref{lem:Duality}.
\end{proof}

We define a transition semigroup $(X, S, \mathbb{P})$ to be prime if $(X, S)$ is prime as a transformation semigroup. Theorem \ref{thm:Prime} provides a way to classify any set of stochastic matrices. If $T$ is any semigroup of $\mathrm{S}(n, \mathbb{R})$, then $S = \mathrm{supp}(T)$ is a set of row monomial binary matrices isomorphic to a subsemigroup of $F_n$. Set $\mathbf{n} = X$. Then each matrix in $T$ is an instance in $(X, S, \mathbb{P})$.

\section{Representation Theory of Probabilistic Autom\-ata} \label{sec:5}

\subsection{Munn-Ponizovski\u{\i} Theory}

Let $A$ be an associative algebra with unity. We denote by $\mathbf{Mod}\text{-}A$ the category of right $A$-modules. Put $J = \mathrm{Rad}(A)$. For any primitive idempotent $e$ of $A$, $eJ$ is the unique maximal submodule of $eA$ in $\mathbf{Mod}\text{-}A$. Assume further that $A$ is noetherian or artinian. This ensures that there exists a collection of pairwise orthogonal central idempotents $e_1, \cdots, e_n \in A$ such that $1_A = e_1 + \cdots + e_n$, or equivalently,
	\begin{equation*}
		A_A = e_1A \oplus \cdots \oplus e_nA.
	\end{equation*}
Moreover, $M \in \mathbf{Mod}\text{-}A$ is simple if and only if $M \cong e_iA/e_iJ$ for some $1 \leq i \leq n$, and hence there is a one-to-one correspondence between isomorphism classes of irreducible modules and that of principal indecomposable modules.

For any idempotent $e$ of $A$, set $B = eAe$. Then $B$ is a subalgebra of $A$. We define restriction as the covariant functor $\mathrm{Res}_B^A: \mathbf{Mod}\text{-}A \to \mathbf{Mod}\text{-}B$ given by
	\begin{equation*}
		\mathrm{Res}_B^A(M) = Me
	\end{equation*}
and induction as its left adjoint functor $\mathrm{Ind}_B^A: \mathbf{Mod}\text{-}B \to \mathbf{Mod}\text{-}A$ given by
	\begin{equation*}
		\mathrm{Ind}_B^A(M) = M \otimes_B eA.
	\end{equation*}
Then $\mathrm{Res}_B^A$ is exact and $\mathrm{Ind}_B^A$ is left exact.

\begin{theorem}[Green] \label{thm:Green}
Let $e \neq 0$ be an idempotent of an associative algebra $A$.
	\begin{compactenum}
		\item If $M \in \mathbf{Mod}\text{-}A$ is simple, then $\mathrm{Res}_{eAe}^A(M) \in \mathbf{Mod}\text{-}eAe$ is either trivial or simple.
		\item If $N \in \mathbf{Mod}\text{-}eAe$ is simple, then the quotient of $\mathrm{Ind}_{eAe}^A(N)$ by its unique maximal submodule
			\begin{equation*}
				\left\{m \in \mathrm{Ind}_{eAe}^A(N) \mid mAe = 0\right\}
			\end{equation*}
		is the unique simple $M \in \mathbf{Mod}\text{-}A$ such that $\mathrm{Res}_{eAe}^A(M) = N$.
	\end{compactenum}
Consequently, there is a one-to-one correspondence between simple $A$-modules that are not annihilated by $e$ and simple $B$-modules.
\end{theorem}

Around the same time, Munn \cite{M2} \& Ponizovski\u{\i} \cite{P} independently furthered the work of Clifford \cite{C} by characterizing irreducible representations of a finite semigroup by those of its principal factors. Lallement \& Petrich \cite{LP}, and later Rhodes \& Zalcstein \cite{RZ}, provided a precise construction based on Theorem \ref{thm:Rees}. We closely follow the arguments of Ganyushkin, Mazorchuk \& Steinberg \cite{GMS} in which the same results are recovered by virtue of Theorem \ref{thm:Green}.

Let $S$ be a finite semigroup. For a field $K$, $KS$ is artinian, so that the notions of semisimplicity and semiprimitivity coincide. It is evident that $KS$ need not be semisimple. Consider, for instance, $K\bar{X}$ for any finite set $X$. For $M \in \mathbf{Mod}\text{-}KS$, we denote by $\mathrm{Ann}_S(M)$ the ideal of $S$ consisting of elements that annihilate $M$.

\begin{definition}
Let $M \in \mathbf{Mod}\text{-}KS$, where $K$ is a field and $S$ a finite semigroup. If $e$ is an idempotent of $S$ satisfying
	\begin{equation*}
		\mathrm{Ann}_S(M) = \{s \in S \mid J_e \subset J(s)\},
	\end{equation*}
then $J_e$ is said to be the apex of $M$.
\end{definition}

Suppose $M \in \mathbf{Mod}\text{-}KS$ is simple. Then there exists a unique apex $J_e$ of $M$. Set $I = \mathrm{Ann}_S(M)$. We identify $M$ with the unique simple $N \in \mathbf{Mod}\text{-}KS/KI$ such that $Ne \neq 0$. By Proposition \ref{prop:Identities},
	\begin{equation*}
		e(KS/KI)e \cong K(eSe)/K(eIe) \cong KH_e.
	\end{equation*}
Let $E(S)$ be a collection of idempotent class representatives of regular $\mathfrak{j}$-classes of $S$. We also write $\mathrm{Res}_{H_e}^S(M)$ and $\mathrm{Ind}_{H_e}^S(M)$, respectively, to mean the restriction and induction functors.

\begin{theorem}[Munn-Ponizovski\u{\i}] \label{thm:MunnPonizovskii}
Let $K$ be a field. Suppose $e \in E(S)$, where $S$ is a finite semigroup.
	\begin{compactenum}
		\item If $M \in \mathbf{Mod}\text{-}KS$ is simple with apex $J_e$, then $\mathrm{Res}_{H_e}^S(M) \in \mathbf{Mod}\text{-}KH_e$ is simple.
		\item If $N \in \mathbf{Mod}\text{-}KH_e$ is simple, then the quotient of $\mathrm{Ind}_{H_e}^S(N)$ by its unique maximal submodule
			\begin{equation*}
				\left\{m \in \mathrm{Ind}_{H_e}^S(N) \mid mKSe = 0\right\}
			\end{equation*}
		is the unique simple $M \in \mathbf{Mod}\text{-}KS$ with apex $J_e$ such that $\mathrm{Res}_{H_e}^S(M) = N$.
	\end{compactenum}
Consequently, there is a one-to-one correspondence between irreducible representations of $S$ and those of $H_e$ for $e \in E(S)$.
\end{theorem}

Again, by Proposition \ref{prop:Identities}, we know $e(KS/KI) \cong R_e$, from which it follows that
	\begin{equation*}
		\mathrm{Ind}_{H_e}^S(N) \cong N \otimes_{KH_e}KR_e
	\end{equation*}
for any $N \in \mathbf{Mod}\text{-}KH_e$, where $e \in E(S)$.

Sch\"utzenberger \cite{S1, S2} studied the action of $S$ on $L_s$ and $R_s$ for any $s \in S$. First define $\Lambda(H_s)$ to be the quotient of the right action of the monoid
	\begin{equation*}
		\{u \in S^1 \mid uH_s \subset H_s\}
	\end{equation*}
on $H_s$ by its kernel. Then $\Lambda(H_s)$ is isomorphic to the group of all maps of the form $\lambda_u|_{H_s}: H_s \to H_s$, and acts freely on $R_s$ from the left. We call $\Lambda(H_s)$ the \emph{left Sch\"utzenberger group} of $H_s$. Its orbit space $\Lambda(H_s) \backslash R_s$ consists of $\mathfrak{h}$-classes in $R_s$. Moreover, $\Lambda(H_s) \cong \Lambda(H_t)$ if $s \sim_\mathfrak{l} t$. A dual statement holds for the \emph{right Sch\"utzenberger group} $\Gamma(H_s)$. In particular, $\Lambda(H_s) \cong \Gamma(H_s)^\mathrm{op}$.

Suppose $\Lambda(H_s) \backslash R_s$ consists of $n$ number of $\mathfrak{h}$-classes. Choose a class representative for each $\mathfrak{h}$-class, so that we can write
	\begin{equation*}
		\Lambda(H_s) \backslash R_s = \{H_{s_1}, \cdots, H_{s_n}\}.
	\end{equation*}
Let $1 \leq i \leq n$. Given $t \in S$, if $s_it \in R_s$, then $s_it \in H_{s_j}$ for some $1 \leq j \leq n$, and so there exists $h \in \Lambda(H_s)$ such that $s_it = hs_j$. The \emph{right Sch\"utzenberger representation} is a map $\rho: S \to \mathrm{M}_n(\Lambda(H_s))$ defined by
	\begin{equation*}
		\rho(t)_{ij} =
		\begin{cases}
			h & \text{if } s_it = hs_j, \\
			0 & \text{otherwise}.
		\end{cases}
	\end{equation*}
The dual construction leads to the \emph{left Sch\"utzenberger representation} $\lambda: S \to M_n(\Gamma(H_s))$.

\subsection{Holonomy Decomposition}

The original proof of Theorem \ref{thm:KrohnRhodes} by Krohn \& Rhodes \cite{KR} is purely algebraic. Based on the work of Zeiger \cite{Z1, Z2}, Eilenberg \cite{E} devised a decomposition that retains the combinatorial structure of a transformation semigroup.

Let $(X, S)$ be a transformation semigroup. We can extend the action of $S$ on $X$ to $S^1$ by requiring that $x1 = x$ for any $x \in X$. Set
	\begin{equation*}
		XS = \{Xs \mid s \in S^1 \cup \bar{X}\} \cup \{\emptyset\}.
	\end{equation*}
Write $a \leq b$ if $a \subset bs$ for some $s \in S^1$. Then the quasiorder $\leq$ induces an equivalence relation $\sim$ given by $a \sim b$ if and only if $a \leq b$ and $b \leq a$. We write $a< b$ to mean $a \leq b$ and not $b \leq a$. A height function is a map $\eta: XS \to \mathbb{Z}$ satisfying
	\begin{compactenum}
		\item $\eta(\emptyset) = -1$,
		\item $\eta(x) = 0$ if $x \in X$,
		\item $a \sim b$ implies $\eta(a) = \eta(b)$,
		\item $a < b$ implies $\eta(a) < \eta(b)$,
		\item $\eta(a) = i$ for some $a \in XS$ if $0 \leq i \leq \eta(X)$.
	\end{compactenum}
The height of $(X, S)$, denoted $\eta(X, S)$, is defined as $\eta(X)$. We can always define a height function on $XS$ by assigning $\eta(a) = i$, where $a_0 < \cdots <a_i$ is a maximal chain in $XS$ such that $a_0 \in X$ and $a_i = a$.

Assume $|a| > 1$ for $a \in XS$. Consider the set $X_a$ of all maximal proper subsets of $a$ contained in $XS$. We call an element of $X_a$ a  \emph{brick} of $a$.  If $as = a$, then $X_as = X_a$, so that $s$ permutes $X_a$. Let $G_a$ denote the coimage of
	\begin{equation*}
		\{s \in S \mid as = s\} \to \mathrm{Sym}(X_a).
	\end{equation*}
Clearly, $G_a \prec S$. If $G_a \neq \emptyset$, $(X_a, G_a)$ is a transformation group. Furthermore, $a \sim b$ implies $(X_a, G_a) \cong (X_b, G_b)$. In case $G_a = \emptyset$, put $G_a = 1$.

Suppose $\eta$ admits $j$ elements, say $a_1, \cdots, a_j$, of height $k$ in ${XS}/{\sim}$. Then we call $X_k = X_{a_1} \times \cdots \times X_{a_j}$ the $k$th \emph{paving} and $G_k = G_{a_1} \times \cdots \times G_{a_j}$ the $k$th \emph{holonomy group}. The $k$th \emph{holonomy} is the transformation semigroup
	\begin{equation*}
		\mathrm{Hol}_k(X, S) = \overline{(X_k, G_k)}.
	\end{equation*}
This is well-defined since $G_k$ is independent of the choice of $a_1, \cdots, a_j$ in ${XS}/{\sim}$.

\begin{theorem}[Eilenberg] \label{thm:Eilenberg}
If $(X, S)$ is a transformation semigroup with a height function $\eta: XS \to \mathbb{Z}$ such that $\eta(X, S) = n$, then
	\begin{equation*}
		(X, S) \prec \mathrm{Hol}_1(X, S) \wr \cdots \wr \mathrm{Hol}_n(X, S),
	\end{equation*}
where $\mathrm{Hol}_i(X, S)$ is the $i$th holonomy for $1 \leq i \leq n$.
\end{theorem}

The decomposition in Theorem \ref{thm:Eilenberg} is known as the \emph{holonomy decomposition} of $(X, S)$ induced by $\eta$. For brevity, we write
	\begin{equation*}
		\mathrm{Hol}_*(X,S) = \mathrm{Hol}_1(X, S) \wr \cdots \wr \mathrm{Hol}_n(X, S).
	\end{equation*}
Since $\bar{\mathbf{n}}^1$ embeds in $n$ direct copies of $\bar{\mathbf{2}}^1$, applying Theorem \ref{thm:JordanHolder} to Theorem \ref{thm:Eilenberg} indeed leads to a prime decomposition of $(X, S)$. If $\mathrm{Hol}_*(X, S) = (Y, T)$, then $T$ is called the \emph{holonomy monoid} of $(X, S)$.

\begin{definition}
Let $(X, S)$ and $(Y, T)$ be transformation semigroups. If there exists a surjective relation $\varphi: Y \to X$ such that for every $s \in S$,
	\begin{equation*}
		\varphi s \subset t \varphi
	\end{equation*}
for some $t \in T$, then $(Y, T)$ is said to cover $(X, S)$ by $\varphi$. We write
	\begin{equation*}
		(X, S) \prec_\mathrm{rel} (Y, T)
	\end{equation*}
to mean $(Y, T)$ is a cover of $(X, S)$, and refer to $\varphi$ as a relational covering.
\end{definition}

If $Y\varphi \subset XS$, then the \emph{rank} of $\varphi$ is the smallest integer $k \geq 0$ such that $\eta(y \varphi) \leq k$ for all $y \in Y$. Note that $(X, S)$ divides $(Y, T)$ when $\varphi$ is of rank $0$.

\begin{proof}[Sketch of proof of Theorem \ref{thm:Eilenberg}]
It suffices to show that if $\varphi: Y \to X$ is of rank $k$, then there exists a map $\psi: X_k \times Y \to X$ of rank $k - 1$ such that
	\begin{equation*}
		(X, S) \prec_\mathrm{rel} \mathrm{Hol}_k(X, S) \wr (Y, T)
	\end{equation*}
by $\psi$, for $\mathbf{1}^1$ covers $(X, S)$ by the unique relation $\mathbf{1} \to X$ of rank $n$.

Let $a_1, \cdots, a_j$ represent elements of height $k$ in ${XS}/{\sim}$. If $\eta(y\varphi) = k$, then $y\varphi \sim a_i$ for a unique $1 \leq i \leq j$, so that we can find $u_y, v_y \in S$ such that
	\begin{equation*}
		a_i{u_y} = y\varphi \text{ and } y\varphi{v_y} = a_i.
	\end{equation*}
Assume such a selection has been made for all $y \in Y$ such that $\eta(y\varphi) = k$. We write a projection map as $\pi_i: (X_k, G_k) \to (X_{a_i}, G_{a_i})$. Define $\psi: X_k \times Y \to X$ by
	\begin{equation*}
		(b, y)\psi =
		\begin{cases}
			y\varphi & \text{if } \eta(y\varphi) < k, \\
			b{\pi_i}u_y & \text{if } y\varphi \sim a_i.
		\end{cases}
	\end{equation*}
It is easy to see that $\psi$ is of rank $k - 1$ with $\mathrm{Im}(\psi) \subset XS$.

Fix $s \in S$. It remains to prove that there exists $(f, t) \in (G_k \cup \bar{X}_k)^Y \rtimes T$ such that the diagram
	\begin{center}
		\begin{tikzpicture}
			\node (0) {$X_k \times Y$};
			\node (1) [below of=0] {$X$};
			\node (2) [right of=1, xshift=20pt] {$X$};
			\node (3) [above of=2] {$X_k \times Y$};
			\draw[->] (0) to node [swap] {$\psi$} (1);
			\draw[->] (1) to node [swap] {$s$} (2);
			\draw[->] (3) to node {$\psi$} (2);
			\draw[->] (0) to node {$(f, t)$} (3);
		\end{tikzpicture}
	\end{center}
commutes. Choose any $t \in T$ satisfying $\varphi{s} \subset t\varphi$. We can find a map $f: Y \to G_k \cup \bar{X}_k$ such that if $y\varphi \sim a_i$, then
	\begin{equation*} \label{map_condition}
		f\pi_i =
		\begin{cases}
			u_ysv_{yt} & \text{if } y\varphi{s} = yt\varphi, \\
			\bar{b}_i & \text{if } y\varphi{s}v_{yt} \subset b_i \text{ with } b_i \in X_{a_i}.
		\end{cases}
	\end{equation*}
It is routine to check that $\psi{s} \subset (f, t)\psi$.
\end{proof}

Given $t \in T$, $t_i$ denotes the $i$th component of $t$. In particular, if $1 \leq i < n$, then $t_i$ is a map $X_{i + 1} \times \cdots \times X_n \to G_i \cup \bar{X}_i$. Suppose that if either
	\begin{compactenum}
		\item there exists $(x_{k + 1}, \cdots, x_n) \in X_{k + 1} \times \cdots \times X_n$ such that $(x_{k + 1}, \cdots,  x_n)t_k \in G_k$ for some $1 < k < n$,
		\item $t_n \in G_n$ with $k = n$,
	\end{compactenum}
then $(x_{i + 1}, \cdots, x_n)t_i \in G_i$ for all $1 \leq i < k$. Then $t$ is said to satisfy the \emph{Zeiger property}.

\begin{lemma} \label{lem:ZeigerProperty}
Suppose $(X, S)$ is a transformation semigroup with a height function $\eta: XS \to \mathbb{Z}$ such that $\eta(X, S) = n$, which admits a decomposition
	\begin{equation*}
		\mathrm{Hol}_*(X, S) = (Y, T).
	\end{equation*}
Then the set $U$ of elements of $T$ satisfying the Zeiger property forms a submonoid of $T$ such that $(Y, U)$ covers $(X, S)$.
\end{lemma}

\begin{proof}
It is easy to see that $U$ is indeed a monoid. Assume $(x_{k + 1}, \cdots, x_n)t_k \in G_k$ for $1 < k < n$. By construction,
	\begin{equation*}
		(x_k, \cdots, x_n)\varphi{s} = (x_k, \cdots, x_n)(t_k, \cdots, t_n)\varphi,
	\end{equation*}
where $\varphi: X_k \times \cdots \times X_n \to X$ is a relation of rank $k - 1$ such that
	\begin{equation*}
		(X, S) \prec_\mathrm{rel} \mathrm{Hol}_k(X, S) \wr \cdots \wr \mathrm{Hol}_n(X, S)
	\end{equation*}
by $\varphi$. If $a_1, \cdots, a_j$ are elements of height $k$ in ${XS}/{\sim}$, then $(x_{k + 1}, \cdots, x_n)\varphi \sim a_i$ for some $1 \leq i \leq j$. Define $t_{k - 1}: X_k \times \cdots \times X_n \to G_{k - 1}$ by
	\begin{equation*}
		(x_k, \cdots, x_n)t_{k - 1}\pi_k = u_{(x_{k + 1}, \cdots, x_n)}sv_{(x_{k + 1}, \cdots, x_n)}.
	\end{equation*}
Put $t_{k - 1}\pi_i = 1_{G_{a_i}}$ for $i \neq k$. The case when $k = n$ is similar.
\end{proof}

A height function $\eta$ uniquely determines $U$, which is referred to as the \emph{reduced holonomy monoid} of $(X, S)$. We also write
	\begin{equation*}
		\widetilde{\mathrm{Hol}}_*(X, S) = (Y, U),
	\end{equation*}
and call $(Y, U)$ the \emph{reduced holonomy decomposition} of $(X, S)$ induced by $\eta$.

\subsection{Representation Theory of Reduced Holonomy Monoid}

Suppose a height function $\eta: XS \to \mathbb{Z}$ on a transformation semigroup $(X, S)$ such that $\eta(X, S) = n$ induces the reduced holonomy decomposition
	\begin{equation*}
		\widetilde{\mathrm{Hol}}_*(X, S) = (Y, U).
	\end{equation*}
We wish to study the representation theory of the transition monoid $(Y, U, \mathbb{P})$. Since $\mathbb{P}U$ does not have an additive structure, we apply Theorem \ref{thm:MunnPonizovskii} to $\mathbb{C}U$, and consider the inclusion $\mathbb{P}U \hookrightarrow \mathbb{C}U$.

The depth function on $U$ is a map $\delta: U \to \mathbb{Z}$ such that for $u \in U$, $\delta(u) = k$ if there exists $0 \leq k \leq m$ satisfying
	\begin{compactenum}
		\item $\mathrm{Im}(u_i) \cap G_i \neq \emptyset$ for $1 \leq i \leq k$,
		\item $\mathrm{Im}(u_i)$ is a singleton in $\bar{X}_i$ for $k < i \leq n$,
	\end{compactenum}
and $\delta(u) = -1$ otherwise. The depth of $(X, S)$ is the largest integer $-1 \leq m \leq n$ such that $\delta(u) = m$ for some $u \in U$. We refer to the pair $(m, n)$ as the \emph{dimension} of $(X, S)$, and write $\dim(X, S) = (m, n)$.

\begin{proposition} \label{prop:RegularityCondition}
Let $(X, S)$ be a transformation semigroup with height function $\eta: XS \to \mathbb{Z}$, which induces a reduced holonomy decomposition
	\begin{equation*}
		\widetilde{\mathrm{Hol}}_*(X, S) = (Y, U)
	\end{equation*}
such that $\dim(X, S) = (m, n)$.	Then $u \in U$is regular if and only if $\delta(u) = k$ for some $0 \leq k \leq m$. Therefore $e \in U$ such that $\delta(e) = k$ is idempotent in $U$ if and only if
	\begin{compactenum}
		\item $(x_{i + 1}, \cdots, x_n)e_i = 1_{G_i}$ for $1 \leq i \leq k$,
		\item $e_i = \bar{x}_i$ for $k < i \leq n$
	\end{compactenum}
for some $(x_{k + 1}, \cdots, x_n) \in X_{k + 1} \times \cdots \times X_n$.
\end{proposition}

\begin{proof}
If $u \in U$ is regular, there exists $v \in U$ such that $uvu = u$. Fix $1 < k \leq n$. Suppose $\mathrm{Im}(u_{k - 1}) \subset \bar{X}_{k - 1}$ and $\mathrm{Im}(u_k)$ is a singleton in $\bar{X}_k$ for $k \leq i \leq n$. Then
	\begin{equation*}
		u_{k - 1}{^{(u_k, \cdots, u_n)}v_{k - 1}}{^{(u_k, \cdots, u_n)(v_k, \cdots, v_n)}u_{k - 1}} = u_{k - 1},
	\end{equation*}
and so $\mathrm{Im}(u_{k - 1})$ is also a singleton in $\bar{X}_{k - 1}$.

Conversely, assume $u \in U$ with $\delta(u) = k$ for some $1 \leq k \leq m$. This means $(x_{k + 1}, \cdots, x_n)u_k  \in G_k$ for some $(x_{k + 1}, \cdots, x_n) \in X_{k + 1} \times \cdots \times X_n$. We want to find $v \in U$ such that $uvu = u$. Set $v_i = \bar{x}_i$ for $k < i \leq n$. It follows from Lemma \ref{lem:ZeigerProperty} that $(x_{i + 1}, \cdots, x_n)u_i \in G_i$ when $1 \leq i < k$. Therefore there exists $v_i: X_{i + 1} \times \cdots \times X_n \to G_i$ such that
	\begin{equation*}
		v_i{^{(v_{i + 1}, \cdots, v_n)}}{u_i} = 1_{G_i}
	\end{equation*}
for $1 \leq i \leq k$.
\end{proof}

Given $1 \leq k \leq m$, denote by $H_k$ the group acting on $X_1 \times \cdots \times X_k$ for the transformation group
	\begin{equation*}
		(X_1, G_1) \wr \cdots \wr (X_k, G_k).
	\end{equation*}
For fixed $y \in Y$, define
	\begin{equation*}
		E(U, y) = \{e \in U \mid e^2 = e \text{ and } e_i = \bar{y}_i \text{ whenever } e_i \neq 1_{G_i} \text{ for } 1 \leq i \leq n\}.
	\end{equation*}
Then $E(U, y)$ contains exactly one idempotent of depth $k$ for each $1 \leq k \leq m$. We also write
	\begin{equation*}
		Y_i = X_{i + 1} \times \cdots \times X_n
	\end{equation*}
for $0 \leq i \leq n$, so that $Y_0 = Y$ and $Y_n = \emptyset$. Then $H_k \times \bar{Y}_k$ is a subsemigroup of $U$ containing $e \in E(U, y)$ such that $\delta(e) = k$.

\begin{proposition} \label{prop:Class}
Let $(X, S)$ be a transformation semigroup with height function $\eta: XS \to \mathbb{Z}$, which induces a reduced holonomy decomposition
	\begin{equation*}
		\widetilde{\mathrm{Hol}}_*(X, S) = (Y, U)
	\end{equation*}
such that $\dim(X, S) = (m, n)$. Fix $y \in Y$. If $u, v \in U$ are regular with $\delta(u) = k$, then
	\begin{compactenum}
		\item $u \sim_\mathfrak{l} v$ if and only if $\delta(u) = \delta(v)$ and $u_i = v_i$ for every $k < i \leq n$,
		\item $u \sim_\mathfrak{j} v$ if and only if $\delta(u) = \delta(v)$.
	\end{compactenum}
For $1 \leq k \leq m$, if $e \in E(U, y)$ such that $\delta(e) = k$, then
	\begin{compactenum} \setcounter{enumi}{2}
		\item $R_e \cong H_k \times \bar{Y}_k$,
		\item $H_e \cong H_k$.
	\end{compactenum}
\end{proposition}

\begin{proof}
(1) If $u \sim_\mathfrak{l} v$, then it is necessary that $\delta(u) = \delta(v)$, and hence $u_i = v_i$ for $k < i \leq n$. Assume the converse. By Lemma \ref{lem:ZeigerProperty} and Proposition \ref{prop:RegularityCondition}, there is $(x_{k + 1}, \cdots, x_n) \in X_{k + 1} \times \cdots \times X_n$ such that $(x_{i + 1}, \cdots, x_n)u_i \in G_i$ for $1 \leq i \leq k$. Therefore we can find $w \in U$ such that
	\begin{equation*}
		w_i{^{(w_{i + 1}, \cdots, w_n)}u_i} = v_i
	\end{equation*}
for $1 \leq i \leq k$ once we set $w_i = \bar{x}_i$ for $k < i \leq n$. This shows that $wu = v$. By symmetry, we conclude that $u \sim_\mathfrak{l} v$.

(2) Again, $u \sim_\mathfrak{j} v$ implies that $\delta(u) = \delta(v)$. Conversely, if $\delta(u) = \delta(v)$, then $u \sim_\mathfrak{r} ue$ if $e \in U$ such that $\delta(e) = k$ is an idempotent defined by
	\begin{equation*}
		e_i =
		\begin{cases}
			1_{G_i} & \text{for } 1 \leq i \leq k, \\
			v_i & \text{otherwise}. 
		\end{cases}
	\end{equation*}
It follows from (1) that $ue \sim_\mathfrak{l} v$.

(3) Assume $u \sim_\mathfrak{r} e$. By (2), $\delta(u) = k$, which means $u_i$ is a singleton in $\bar{X}_i$ for $k < i \leq n$. Since $ev = u$ for some $v \in U$,
	\begin{equation*}
		e_i{^{(e_{i + 1}, \cdots, e_n)}}{v_i} = u_i
	\end{equation*}
for $1 \leq i \leq k$, which shows that $u_i$ does not depend on $X_{k + 1} \times \cdots \times X_n$. Similarly, $uw = e$ for some $w \in U$, and hence
	\begin{equation*}
		u_i{^{(u_{i + 1}, \cdots, u_n)}}{w_i} = e_i.
	\end{equation*}
Whenever $1 \leq i \leq k$, $\mathrm{Im}(u_i) \subset G_i$ since $e_i = 1_{G_i}$. Therefore we can conclude that $R_e \subset H_i \times \bar{Y}_i$. The opposite inclusion is obvious.

(4) This is an immediate consequence of (1) and (3).
\end{proof}

Proposition \ref{prop:Class} implies that there are exactly $m$ regular $\mathfrak{j}$-classes in $U$ whose maximal subgroup is determined by the first $k$ holonomy groups. We can now apply this to Theorem \ref{thm:MunnPonizovskii} to determine all irreducible representations of $U$.

\begin{theorem}
Let $(X, S)$ be a transformation semigroup with height function $\eta: XS \to \mathbb{Z}$, which induces a reduced holonomy decomposition
	\begin{equation*}
		\widetilde{\mathrm{Hol}}_*(X, S) = (Y, U)
	\end{equation*}
such that $\dim(X, S) = (m, n)$. Fix $y \in Y$. If $K$ is a field, then $M_i \in \mathbf{Mod}\text{-}KU$ satisfying
	\begin{equation*}
		M_i \cong M \otimes_{KH_i} K(H_i \times \bar{Y}_i),
	\end{equation*}
where $M \in \mathbf{Mod}\text{-}KH_i$ is simple and $H_e \cong H_i$ for $e \in E(U, y)$ with $\delta(e) = i$ for $1 \leq i \leq m$, is principal indecomposable. Furthermore, $M_i$ contains a unique maximal submodule
	\begin{equation*}
		N_i = \left\{ m \in M_i \mid mKUe = 0 \right\},
	\end{equation*}
so that $M_i/N_i \in \mathbf{Mod}\text{-}KU$ is simple.
\end{theorem}

\begin{proof}
It is easy to see that elements $m \otimes (1_{H_i}, \bar{z})$, where $m$ is a basis of $M$ and $z \in Y_i$, form a basis of $M_i$. For any $(h, \bar{z}) \in H_i \times \bar{Y}_i$, we can write
	\begin{equation*}
		m \otimes (h, \bar{z}) = m(h, \bar{y}_i) \otimes (1, \bar{z}).
	\end{equation*}
This implies that $M_i$ is indecomposable. Since $M_i$ is free, it is projective, and hence principal indecomposable. The result follows from Theorem \ref{thm:MunnPonizovskii}.
\end{proof}

It follows from Theorem \ref{thm:MunnPonizovskii} that modules of the form $M_i/N_i$ induced by a simple right $KH_i$-module $M$, where $H_e \cong H_i$ for some $e \in E(U, y)$, account for all simple right $KU$-modules.

\end{document}